\documentclass[11pt,reqno]{amsart}
\usepackage{amsmath,amsthm,amssymb,amscd,url,enumerate,mathtools}
\usepackage{graphicx}
\usepackage[margin=1 in]{geometry}
\usepackage{afterpage}
\usepackage[usenames,dvipsnames]{xcolor} 
\usepackage{tikz, tikz-3dplot, pgfplots}
\usepackage{tkz-graph}
\usepgfplotslibrary{fillbetween} 
\usepackage{todonotes}
\usepackage[all]{xy}
\usepackage[colorlinks=true,citecolor={blue},linkcolor = {purple}]{hyperref} 
\usepackage{tabularx}
\usepackage{tabu}
\usepackage{array} 
\usepackage{multirow} 

\usepackage{dutchcal}

\usepackage{amsmath}

\usetikzlibrary{shapes,positioning}

\usepackage[square]{natbib}
\setcitestyle{numbers}

\usepackage{relsize} 

\newcommand{\TITLE}{A Divisor Formula and a 
Bound on the $\mathbb{Q}$-gonality of the Modular Curve $X_1(N)$}
\newcommand{\TITLERUNNING}{}

\theoremstyle{plain}
\newtheorem{theorem}{Theorem}

\newtheorem{proposition}[theorem]{Proposition}
\newtheorem{lemma}[theorem]{Lemma}

\theoremstyle{definition}
\newtheorem{definition}[theorem]{Definition}

\theoremstyle{remark}
\newtheorem{remark}[theorem]{Remark}

\newtheorem{example}[theorem]{Example}

\numberwithin{theorem}{section}

  {\end{list}}

  {\end{list}}

\newcommand{\tightoverset}[2]{%
  \mathop{#2}\limits^{\vbox to -.5ex{\kern-1.05ex\hbox{$#1$}\vss}}}

\newcommand{\e}{\epsilon}

\def\Hcal{{\mathcal H}}

\def\Ocal{{\mathcal O}}

\newcommand{\BB}{\mathbb{B}}
\newcommand{\CC}{\mathbb{C}}

\newcommand{\PP}{\mathbb{P}}
\newcommand{\QQ}{\mathbb{Q}}
\newcommand{\RR}{\mathbb{R}}

\newcommand{\ZZ}{\mathbb{Z}}

\def \bfe{{\mathbf e}}
\def \bff{{\mathbf f}}

\def \bfl{{\mathbf l}}
\def \bfn{{\mathbf n}}

\def \bfp{{\mathbf p}}

\def \bfw{{\mathbf n}}

\newcommand{\GL}{\operatorname{GL}}

\newcommand{\SL}{\operatorname{SL}}

\newcommand{\ord}{\operatorname{ord}}

\newcommand{\divisor}{\operatorname{div}}
\newcommand{\uord}{\operatorname{uord}}

\newcommand{\LPar}{\mathlarger{\mathlarger{((}}}
\newcommand{\RPar}{\mathlarger{\mathlarger{))}}}

\title[\TITLERUNNING]{\vspace*{-1.3cm} \TITLE}

\date{\today}

\author[Mark van Hoeij]{Mark van Hoeij}
\address{%
Dept. of Mathematics, 
Florida State University,
Tallahassee, FL 32306, USA}
\email{hoeij@math.fsu.edu}

\author[Hanson Smith]{Hanson Smith}
\address{%
Dept. of Mathematics,  
University of Colorado,
Campus Box 395,
Boulder, Colorado 80309-0395}
\email{hanson.smith@colorado.edu \text{or} hansonsmith101@gmail.com}


\begin{document}
\begin{abstract}
We give a formula for divisors of modular units on $X_1(N)$ and use it to prove
that the $\mathbb{Q}$-gonality of the modular curve $X_1(N)$ is bounded above by $\left[\frac{11N^2}{840}\right]$, where $[\bullet]$ denotes the nearest integer.
\end{abstract}
\maketitle
\section{Introduction}

The modular curve $X_1(N)$ parametrizes pairs $(E, \pm P)$ where $E$ is an elliptic curve and $P$ is a point of exact order $N$.
As such it has been an object of interest for number theorists and arithmetic geometers. 
If $K$ is a field, then \emph{$K$-gonality} of $X_1(N)$ is the minimum degree of a non-constant function $X_1(N)\to \PP^1$ defined over $K$.

Table~1 in \cite{DerickxvanHoeij} gives the currently-best upper bounds the $\QQ$-gonality of $X_1(N)$ for $N \leq 250$ and matching lower bounds for $N \leq 40$.
Any non-constant function provides an upper bound for the gonality.
The upper bounds in \cite[Table~1]{DerickxvanHoeij} come from {\em modular units}. These are functions on $X_1(N)$ whose divisors
are supported only on {\em cusps}\, (places on $X_1(N)$ where $E$ degenerates).
In this note we prove a formula for the degree of a certain modular unit $F_7/F_8$. 
Its degree is a particularly good gonality bound when $N$ is prime, it is 
currently the best upper bound for all primes $N \leq 250$ except 31, 67, 101, where it is only one more. 

A basis $F_2,F_3,\ldots$ of modular units was given in \cite[Conjecture~1]{DerickxvanHoeij} which was proved in \cite{Streng}.
In order to quickly find the degree of any modular unit, a formula for the divisor of $F_k:  X_1(N) \rightarrow \PP^1$ was given
at \cite{MinFormula}. 
A proof for this formula was not given; the resulting degrees listed in \cite[Table~1]{DerickxvanHoeij} were verified by other means.
The main result in this paper is a proof for this formula (Theorem~\ref{Main} in Section~\ref{4}). As an application, Section~\ref{5} gives
this bound
\[
{\rm Gonality}_{\QQ}\left(X_1(N)\right) \ \leq \ {\rm deg}\left(\frac{F_7}{F_8} : X_1(N) \rightarrow \PP^1\right) \ \leq \  \left[\frac{11N^2}{840}\right] \ \ \ \ \text{if } N>8.
\]
Here $[\bullet]$ indicates rounding to the nearest integer.
The second $\leq$ is an equality when $N$ is prime. The asymptotic growth $11N^2/840 $ was
already observed in \cite[Section~2.1]{DerickxvanHoeij} and \cite[page~11]{SuthNotes} 
(combine the factors 11/35 and $1/24$) though a proof was not given. 

The explicit divisors given in Theorem~\ref{Main} have other applications as well,
such as computing Galois representations for modular curves \cite{GaloisRep},
computing the action of diamond operators \cite{DerickxvanHoeij}, computing cuspidal class numbers of modular curves
(\cite{Carlucci},  \cite{Chen},  \cite[Chapters 5 and 6]{KubertLang},  \cite{Takagi}, \cite{Yang}, \cite{YangJDYu}, \cite{Yoo}, \cite{Yu}),
computing optimized equations for $X_1(N)$ (\cite{Baaziz}, \cite{SutherlandOptimal}, \cite{SutherlandPrescribed}), 
and 
sporadic points on modular curves (\cite{sporadicbound}, \cite{MazurLecture}, \cite{Najman}, \cite{RamificationSporadic}, \cite{SuthNotes}).

Section~\ref{Preliminaries} reviews Puiseux expansions, elliptic curves, and division polynomials.
Notation for modular units is given in Section~\ref{SectionModUnits}.
Section~\ref{4} gives the main theorem. In Section~\ref{5} we obtain the gonality bound as an application of the main theorem. 
Streng \cite{Streng} used Siegel functions to prove \cite[Conjecture~1]{DerickxvanHoeij}.
This work implies another proof for Theorem~\ref{Main}, see Section~\ref{6} for details.
Orders of Siegel functions are typically expressed in terms of Bernoulli polynomials. We observe that such expressions sum to piecewise linear functions (Section~\ref{6.3})
when the corresponding product of Siegel functions is a modular unit.

To save the reader time, Appendix~\ref{NotationSection} tabulates the notations.

\section{Preliminaries} \label{Preliminaries}
\subsection{Places and Puiseux expansions} \label{2.1}
If $f \in \QQ(s)[x]$ is irreducible over $\overline{\QQ}$, then $f$ defines an algebraic curve $C$ whose function field $\QQ(C)$
is $\QQ(s)[x]/(f)$. 

We give a brief summary of Puiseux expansions, see \cite[Chapter II]{SerreLocalFields} for more.
A {\em Puiseux expansion}\, of $f$ at $s=0$ is a root of $f$ in the algebraic closure of $\QQ((s))$.
This is contained in the algebraic closure of $\CC(( s ))$, which is $\bigcup_{e =1}^{\infty} \CC\LPar s^{1/e} \RPar$.
The natural valuation 
\[v_s:  \CC\left(\left( s^{1/e} \right)\right) \rightarrow \frac1e \ZZ \bigcup \{\infty\}\] 
sends a non-zero series
to its lowest exponent in $s$ and sends $0$ to $\infty$.

For a Puiseux expansion $\bfp$, let ${\bf e}_{\bfp}$ be the smallest $e$ for which $\bfp \in \CC\LPar s^{1/e} \RPar$. From the embedding
\[\phi_{\bfp}: \QQ(C) \rightarrow \QQ((s))[\bfp] \subset \CC\left(\left( s^{1/{\bf e}_{\bfp}} \right)\right), \ \ \ \ \ \phi_{\bfp}: x \mapsto \bfp  \]
we get a discrete valuation  
\begin{equation}\label{discretevaluation}
v_{\bfp}: \QQ(C) \rightarrow \ZZ \bigcup \{\infty\} \ \ \text{   given by   } \ \ v_{\bfp}( a ) = {\bf e}_{\bfp} \cdot v_s\left( \phi_{\bfp}(a) \right).
\end{equation}
The factor ${\bf e}_{\bfp}$ in~(\ref{discretevaluation}) ensures that $v_{\bfp}(a)$ lands in $\ZZ\cup \{\infty\}$.
Omitting this factor gives what we will call the {\em unweighted order} $v_s\left(\phi_{\bfp}(a)\right)$ of $a$, which is $1$ at $a=s$ and $1/{\bf e}_{\bfp}$ at a {\em local parameter}. 
%
The {\em residue field} $k_{\bfp}$ is defined as $\left\{ a \in K_{\bfp} \,|\, v_{\bfp}(a) \geq 0\right\}$ modulo $\left\{a \in K_{\bfp} \,|\, v_{\bfp}(a) > 0\right\}$,
where $K_{\bfp} = \QQ((s))[\bfp]$. \\[-8pt] 

A {\em place} on $C/\QQ$ 
is a discrete valuation $v_P: \QQ(C) \rightarrow \ZZ \bigcup \{\infty\}$. A place above $s=0$ is a place with $v_P(s) > 0$.
Puiseux expansions $\bfp$ and $\bfp_1$ are conjugate over $\QQ((s))$ if and only if $v_{\bfp} = v_{\bfp_1}$, so {\em a place 
corresponds to a conjugacy class of Puiseux expansions.} 
A conjugacy class $\{\bfp,\ldots\}$ has $\bfw_{\bfp} := {\bf e}_{\bfp}\,{\bf f}_{\bfp}$ elements, where ${\bf f}_{\bfp} = [k_{\bfp} : \QQ]$.
A valuation $v_{\bfp}:  \QQ(C) \rightarrow \ZZ \bigcup \{\infty\}$ extends to ${\bf f}_{\bfp}$ distinct valuations $\CC(C) \rightarrow \ZZ \bigcup \{\infty\}$, so {\em one place on $C/\QQ$
corresponds to ${\bf f}_{\bfp}$ places on $C/\CC$.}

\begin{example} \label{why} Let $\bfp = c s^{1/2} + \cdots$ where $c \neq 0$ and dots are terms of higher order. Then $v_s(\bfp) = 1/2$ so $\bfp^2/s = c^2 s^0 + \cdots$ has valuation 0
and hence $c^2 \in k_{\bfp}$. However, $c$ need not be in~$k_{\bfp}$.
In that case, to avoid constants not in $ k_{\bfp}$,
we rewrite $c s^{1/2}$ as $(\alpha s)^{1/2}$ where $\alpha = c^2 \in k_{\bfp}$.
\end{example}

\begin{definition} \label{sim}
Let $l_s(\bfp)$ denote the {\em dominant term} (the term with lowest exponent) of a nonzero Puiseux expansion. 
We write $\bfp_1 \sim \bfp_2$ if and only if $l_s\left(\bfp_1\right) = l_s\left(\bfp_2\right)$.
In general, $v_s\left(\bfp_1 - \bfp_2\right) \geq {\rm min}\left(v_s\left(\bfp_1\right), v_s\left(\bfp_2\right)\right)$ with equality if and only if $\bfp_1 \not\sim \bfp_2$.
\end{definition}

Let $P$ be a place above $s=0$ given by a Puiseux expansion $\bfp\in \CC \LPar s^{1/{\bf e}_{\bfp}}\RPar$ 
of $f$.  
Suppose we wish to compute $v_P(g)$ for some $g \in \QQ(s)[x]$.
Write $g = l \left(x - \bfp_1\right) \cdots \left(x - \bfp_n\right)$, where $l \in \QQ(s)$ and the $\bfp_i \in \CC\LPar s^{1/{\bf e}_{\bfp_i}} \RPar$ are the Puiseux expansions of $g$ at $s=0$.
Then $g(\bfp) = l \left(\bfp - \bfp_1\right) \cdots \left(\bfp - \bfp_n\right)$ and $v_P(g) = {\bf e}_{\bfp} \cdot \left(v_s(l) + v_s\left(\bfp-\bfp_1\right) + \cdots +v\left(\bfp-\bfp_n\right)\right)$.
If $\bfp \not\sim \bfp_i$ for each $i$, then:  
\begin{equation} \label{lucky}
v_P(g) = {\bf e}_{\bfp} \cdot \left(\,v_s(l) + {\rm min}\left\{v_s(\bfp), v_s\left(\bfp_1\right)\right\} + \cdots +  {\rm min}\left\{v_s(\bfp), v_s\left(\bfp_n\right)\right\}\,\right).
\end{equation}
\begin{lemma} \label{l} With $g$ and $\bfp_i$ as above, let $\bfl_i := l_s(\bfp_i)$. Suppose that $\bfl_1,\ldots,\bfl_n$ are distinct. Then ${\bf e}_{\bfl_i} = {\bf e}_{\bfp_i}$ and ${k}_{\bfl_i} =  {k}_{\bfp_i}$.
\end{lemma}

\begin{proof} Note ${\bf e}_{\bfl_i} \leq {\bf e}_{\bfp_i}$ and ${k}_{\bfl_i} \subseteq {k}_{\bfp_i}$ because
the ramification index and residue field of $\bfp_i$ must be at least as large as those of its dominant term $\bfl_i$.  If at least one of those is not an equality, then ${\bf n}_{\bfp_i} > {\bf n}_{\bfl_i}$.
In this case, $\bfp_i$ has more conjugates over $\QQ((s))$ than $\bfl_i$, so there must be at least two conjugates with the same dominant term. Those conjugates are among $\bfp_1,\ldots,\bfp_n$
since $g \in \QQ(s)[x]$, which implies that $\bfl_1,\ldots,\bfl_n$ are not distinct.
\end{proof}

\subsection{Elliptic curves, analytic viewpoint} \label{A}
Let $0 < \e \ll 1$ and consider the elliptic curve
\begin{equation} E_{\e}: \ y^2 = x(x-\e)(x-1) \label{Ee} \end{equation}
so $y = \sqrt{x(x-\e)(x-1)}$.
Let $E_{\e}(\CC)$ denote the points on $E$ defined over $\CC$. This is an additive group \cite[Chapter VI]{AEC}, the identity $\Ocal$ is the point at infinity.
The period lattice is $\Lambda = \ZZ \omega_1 + \ZZ \omega_2$ where
\begin{equation} \label{tr1}
\omega_1 = 2 \int_1^{\infty} \frac{{\rm d}x}{y} = 2 \int_0^{\e} \frac{{\rm d}x}{y} = 4\,K\left( \sqrt{\e} \right) = 2\pi \left(1 + \frac14 \e +\frac{9}{64} \e^2 + \frac{25}{256} \e^3 +  \cdots  \right)  \end{equation} and
\begin{equation} \label{tr2}
\omega_2  = 2 \int_{\e}^1 \frac{{\rm d}x}{y} = 2 \int_{-\infty}^0 \frac{{\rm d}x}{y} = \frac{4}{i}\,K\left( \sqrt{1-\e} \right) = \frac{\omega_1}{\pi i} \ln\left( \frac{16}{\e} - 8 - \frac54 \e 
+ \cdots \right) . \end{equation}
Here $K$ is the complete elliptic integral of the first kind~\cite[\textsection 19.2(ii)]{DLMF} 
\[K(t)=\int_0^{\frac{\pi}{2}}\frac{d\theta}{\sqrt{1-t^2\sin^2\theta}}.\]
In this section $\sim$ means that the $\e$-dominant terms are the same, similar to Definition~\ref{sim}. For example
\begin{equation}
	\omega_1 \sim 2 \pi \ \ \text{ and } \ \ \omega_2 \sim \frac{2}{i} \ln\left( \frac{16}{\e} \right).   \label{dom1}
\end{equation}
The notation $\approx$ will be used for approximations in intermediate steps, to indicate that they are sufficiently accurate
to compute the main formulas (\ref{dom1}),(\ref{dom2}),(\ref{dom3}),(\ref{dom5}) up to $\sim$.

The Abel-Jacobi map is an isomorphism (as additive groups) from $E_{\e}(\CC)$ to $\CC/\Lambda$.
Identify $E_{\e}(\CC)/\pm$ with $\PP^1(\CC)$ using $\pm P \mapsto x(P)$. Let $W := (\CC/\Lambda)/\pm$. 
The Abel-Jacobi map (up to $\pm$) is a bijection:
\begin{equation} \label{tr}
	\Psi: \PP^1(\CC) \rightarrow  W, \ \ \text{ where } \ \ \Psi(x_0) = \pm \left( \int_{x_0}^{\infty} \frac{{\rm d}x}{y} \ + \Lambda \right).
\end{equation}
Its inverse is the Weierstrass $\wp$ function \cite[1,II,5,\textsection 1]{HurwitzCourant}.
%
%
%
Each element of  $W$ can be written uniquely as 
\begin{equation} \label{w}
	\pm (r_1 \omega_1 + r_2 \omega_2 + \Lambda), \ \ {\rm with} \ \ r_1 \in [0,1), \ \ r_2 \in \left[0,\frac12\right], \ {\rm and \ if} \ r_2  \in \left\{0, \frac12\right\} {\rm \ then \ } r_1 \in \left[0, \frac12\right].
\end{equation}
Although $W$ is not a group, it inherits the multiplication by $N$ map from $\CC/\Lambda$. 
The order of the element~(\ref{w}) is $N$ if and only if $r_1,r_2 \in \QQ$ and the least common multiple of their denominators is $N$.
The image of $\PP^1(\RR)$ under $\Psi$ is a rectangle in $W$ whose corners are the points of order 1 and 2.

Like in the modular description in Section \ref{Section6.1},
define the \emph{Cartan} as $C(N) := \{ 0,\ldots,\lfloor N/2 \rfloor \}$. 
Let $W(N) \subset W$ be the set of elements of order $N$, 
and for each $c \in C(N)$ let $W_c(N) \subseteq W(N)$ be the subset where $r_2 = c/N \in [0,1/2]$.
Let $\bfw_c(N) := |W_c(N)|$ denote the cardinality of $W_c(N)$. \\[-8pt]

\begin{enumerate}
\item[$\bullet$]Case $c = 0$. Then  $\bfw_0(2) = 1$ and $\bfw_0(N) = {\varphi(N)}/{2}$ for $N>2$. \\[-8pt]
\item[$\bullet$]Case $0 < c < N/2$. Then $\bfw_c(N) = {\varphi(d) N}/{d}$, where $d = {\rm gcd}(c,N)$. \\[-8pt]
\item[$\bullet$]Case $c = N/2$. Then $\bfw_{1}(2) = 2$ and $\bfw_{N/2}(N) =\varphi\left(\frac{N}{2}\right)$ for even $N>2$.  \\[-8pt]  
\end{enumerate}

For later use we define ${\bf e}_c(N), {\bf f}_c(N)$ with these formulas: $\bfw_c(N) = {\bf e}_c(N) \cdot {\bf f}_c(N)$ where
${\bf e}_2(4) := 1$ and ${\bf e}_c(N) := N/d$ otherwise. 
Define $C_c(N) := \Psi^{-1}( W_c(N) ) \subset \PP^1(\CC)$ so that
\[
	\bigcup_{c \in C} C_c(N)= \Psi^{-1}(W(N)) = \left\{\, x(P)\ |\ P \in E_{\e}(\CC) {\rm \ has \ exact \ order \ }N \,\right\}.
\]
To prove the main theorem in Section~\ref{4} it suffices to compute these $x(P)$'s up to $\sim$.
We find $C_0(2) = \{1\}$ and $C_1(2) = \{ 0, \e \}$ from the definition.
%
%
%
Next we compute $C_0(N)$ up to $\sim$ for $N \geq 2$. 
For $C_0(N)$ we have $r_2 = 0$ and $x(P) \in [1,\infty)$. Let $y_1 = x \sqrt{x-1}$. If $\e \ll |x|$ then $y \approx y_1$.
For any $x_0 \in [1, \infty)$ we have
\begin{equation} \label{arctan}
	\Psi(x_0) = \int_{x_0}^{\infty} \frac{{\rm d}x}y \ \approx \ \int_{x_0}^{\infty} \frac{{\rm d}x}{y_1} = \pi - 2\, {\rm arctan}\left( \sqrt{x_0 - 1}\right).
\end{equation}
Equating \eqref{arctan} to $r_1 \omega_1 + 0 \omega_2$ gives $x_0 \sim \sin(\pi r_1)^{-2} = 2/(1-\cos(2 \pi r_1))$.  
Substituting $r_1 =a/N$ gives
\begin{equation} \label{dom2}
	C_0(N) \sim \left\{ \sin\left(\frac{a \pi}{N}\right)^{-2}  \ \middle| \ 0<a\leq \frac{N}2, \ {\rm gcd}(a,N) = 1 \right\}. 
\end{equation}
For $C_{N/2}(N)$ we have $r_2 = 1/2$ and $x(P) \in [0,\e]$.
Let $y_0 = \sqrt{x(x-\e)(-1)}$. If $|x| \ll 1$ then $y \approx y_0$.
Let  $x_0 \in [0,\e]$. Working mod $\Lambda = \ZZ \omega_1 + \ZZ \omega_2$, see~(\ref{tr}) and~(\ref{tr1}),(\ref{tr2}), we have 
\begin{equation} \label{arcsin}
	\Psi(x_0)=\int_{-\infty}^{x_0}  \frac{{\rm d}x}y = \frac{\omega_2}2 + \int_0^{x_0}  \frac{{\rm d}x}y \ \approx \ \frac{\omega_2}2 + \int_0^{x_0}  \frac{{\rm d}x}{y_0}
	=\frac{\pi + \omega_2}2 - \arcsin\left(1 - \frac{2{x_0}}{\e}\right).
\end{equation}
Equating \eqref{arcsin} to $r_1 \omega_1 + \frac12 \omega_2$ gives $x_0 \sim \e \cdot \sin(\pi r_1)^2 $. Substituting $r_1 = a/N$ gives
\begin{equation} \label{dom3}
	C_{\frac{N}2}(N) \sim \left\{\,\e \cdot \sin\left( \pi \cdot \frac{a}N\right)^2  \ \middle| \ 0 \leq a \leq \frac{N}2, \ {\rm gcd}\left(a,\frac{N}2\right) = 1 \right\}. 
\end{equation}

Now let $r_2 \in (0,1/2)$ which corresponds to $\e \ll |x_0| \ll 1$ under $\e \rightarrow 0^+$. 
By equating the right-hand side of~(\ref{arctan}), or that of~(\ref{arcsin}), to $r_1 \omega_1 + r_2 \omega_2$ 
and computing a series expansion we find
\[
	x_0 \sim -4 e^{- 2\pi i  r_1} \left(\frac{\e}{16} \right)^{2 r_2}.
\]
Substituting $r_1 = -a/N$ (the minus sign does not affect~(\ref{dom5})) and $r_2 = c/N$ gives
\begin{equation} \label{dom5}
	C_c(N) \sim \left\{ -4 \, \zeta_N^{a} \left(\frac{\e}{16} \right)^{\frac{2 c}N} \ \middle| \ 0 \leq a < N, \ {\rm gcd}(a,c,N) = 1 \right\}.
\end{equation}
The $a$ and $c$ in \eqref{dom5} are the $a$ and $c$ appearing in the vectors in Subsection \ref{Section6.1}. After rewriting $\e$ in terms of $s$ from Section \ref{4}, Equation \eqref{dom5} is enough to determine the Galois action.

\subsection{Division polynomials} \label{SecDiv}
Let $K$ be a field of characteristic 0 and take $a,b \in K$ for which
\begin{equation} E: \ y^2 = x^3 + ax + b \label{defE} 
\end{equation}
defines an elliptic curve over $K$. Following \cite[Exercise 3.7]{AEC}, the division polynomials $Q_k \in \ZZ[x,y,a,b]$, $k = 1,2,\ldots$
are defined by 
\[ Q_1 := 1, \ \ \ Q_2 := 2y = 2\sqrt{x^3 + ax + b}, \ \ \ Q_3 :=  3x^4 + 6ax^2 + 12bx - a^2, \]
\[ Q_4 := 4y(x^6 + 5ax^4 + 20bx^3 - 5a^2x^2 - 4abx - 8 b^2 - a^3), \]
and the recursion relations
\begin{equation} Q_{2k+1} = Q_{k+2} Q_k^3 - Q_{k-1} Q_{k+1}^3 \ \ \ {\rm for} \ k \geq 2 \label{RedOdd} \end{equation}
\begin{equation} Q_{2k} = \frac{(Q_{k+2} Q_{k-1}^2 - Q_{k-2} Q_{k+1}^2) Q_k }{ Q_2} \ \ \ {\rm for} \ k \geq 3. \label{RecEven} \end{equation}
Recursively define $q_k$ to be $Q_k$ divided all $q_d$ with $d | k$ and $d < k$, so that $Q_k = \prod_{d | k} q_d$. One has $q_1=1$, $q_2=Q_2$, $q_3=Q_3$, $q_4=Q_4/Q_2$, and so forth.

Division polynomials have the following properties:
\begin{enumerate}
\item $Q_k$ is in $\ZZ[x,a,b]$ when $k$ is odd, and in  $q_2 \cdot \ZZ[x,a,b]$ when $k$ is even.
\item Let $\Ocal$ be the identity in $(E(\overline{K}),+)$,
let $E[k]$ be the points $P$ in $E(\overline{K})$ with $k P = \Ocal$.
Then $Q_k$ has one pole, of order $k^2-1$ at $\Ocal$,
and a root of order 1 at every $P \in E[k]  - \{\Ocal \}$. The roots of $q_k$ are the points of exact order $k$, denoted $E[=\hspace{-2pt}k] \subseteq E[k]$.
\item The $\pm$ below means: choose only one element of each pair $\{P,-P\} \subset E[k]$  (this is not relevant for $k=2$ because $P = -P$ when $P \in E[2]$).
\begin{equation}\label{prodab}
{\rm If} \ k \ {\rm is \ odd:} \ \ \ Q_k = k\prod_{P\in (E[k]- \{\Ocal \})/\pm} \left(x-x(P)\right) \ \in \QQ[x, a,b].
\end{equation}
\begin{equation}\label{prodabeven}
{\rm If} \ k \ {\rm is \ even:} \ \ \ \frac{Q_k}y=k\prod_{P\in (E[k]-E[2])/\pm} \left(x-x(P)\right) \ \in \QQ[x, a,b].
\end{equation}
\begin{equation*}
{\rm If} \ k>2: \ \ \ q_k = a_k \prod_{P \in E[=k]/\pm} \left(x-x(P)\right) \ \in \QQ[x, a,b],
\end{equation*}
where $a_k = p$ if $k$ is a prime power and $1$ otherwise.
\begin{equation*}
{\rm For} \ k=2: \ \ \  q_2^2 = 4 \prod_{P\in E[=2]} (x-x(P)) \ \in \QQ[x, a,b]. \end{equation*}
\end{enumerate}

The formulas imply that $Q_k$ is square-free, and if $d | k$, then $Q_d | Q_k$. 
Let $m_k$ denote the number of elements of $E[=\hspace{-2pt}k]$. We have
$m_2 = 3$, $m_3 = 8$, and $12 | m_k$ when $k>3$. Note that ${\rm deg}_x(q_k) = m_k/2$.

\section{Equations for $X_1(N)$ and modular units}\label{SectionModUnits}

The definitions of $Q_k$ and $q_k$ are not completely canonical; 
recurrence relations~(\ref{RedOdd}) and (\ref{RecEven}) are preserved
under {\em scaling}. Scaling means multiplying $Q_k$ by $\alpha^{k^2-1}$, and $q_k$ by $\alpha^{m_k}$, for some fixed $\alpha \neq 0$.
To obtain expressions that are independent of scaling, we take quotients
\begin{equation}\label{tildeq}
	\tilde{Q}_k = \frac{Q_k}{q_2^{(k^2-1)/3}} \ \ {\rm and} \ \ 
	\tilde{q}_k = \frac{q_k}{q_2^{m_k/3}}.  
\end{equation}
As before we have $\tilde{Q}_k = \prod_{d|k} \tilde{q}_d$. Since $\tilde{Q}_k=\tilde{q}_k = 1$ for $k \in \{1,2\}$, we have $\tilde{Q}_k=\tilde{q}_k$ for $k < 6$.
To avoid the fractional exponent in \eqref{tildeq}, we also introduce
\begin{equation} \label{F3}
	F_3 = \tilde{q}_3^3 =  \frac{q_3^3}{q_2^8} %
	\ \ {\rm and} \ \ F_k = \tilde{q}_k 
	\ {\rm for} \ k > 3.
\end{equation}
Let $\tilde{Q}_{k \setminus 3}$ be $\tilde{Q}_k / \tilde{q}_3$ if $3 | k$ and $\tilde{Q}_k$ otherwise.
Because $\tilde{Q}_k$ comes from $Q_k$ by scaling, it satisfies
the recurrence relations. These relations inductively show 
that
\begin{equation}\label{prodqtilde}
\tilde{Q}_{k \setminus 3} = \prod_{3 \neq d|k} \tilde{q}_d  \ \in \ZZ[F_3, F_4].
\end{equation}
Assuming that $F_3,F_4 \in \QQ(x,a,b)$ are algebraically independent over $\QQ$, 
the Appendix shows~\eqref{prodqtilde}, and
that $\tilde{Q}_{k \setminus 3}$ is \emph{primitive} in $\ZZ[F_3,F_4]$, i.e.  the gcd 
of the coefficients in $\ZZ$ is 1.
%
%
%
%
The product in \eqref{prodqtilde} is square-free since (\ref{prodab}) and (\ref{prodabeven}) are square-free. 
Then by induction $F_4, F_5, F_6, \ldots$ from \eqref{F3} are primitive, co-prime, and square-free in $\ZZ[F_3,F_4]$. \\[-5pt]

%
%
\noindent {\bf Henceforth}, $E$ will be the curve
\begin{equation} \label{E}
	E: y^2 = x^3 - 3 j_0 x - 2 j_0, \ \ \ {\rm where} \ \ j_0 := \frac{j}{j-1728}.
\end{equation}
Now $E$ is defined over $K = \QQ(j)$, where $j$ is transcendental over $\QQ$. So $a$ and $b$ from Section~\ref{SecDiv} will be $-3 j_0$ and $-2 j_0$ from here on.
Now $q_3, q_4, \ldots$ are in $\ZZ[x,j_0]$. 
The $j$-invariant of $E$ is~$j$, and the $j$-invariant of $E_{\epsilon}$ is
\begin{equation}\label{jepsilon}
\frac{2^8(\e^2-\e+1)^3}{\e^2(\e-1)^2}.   
\end{equation}
In Section~\ref{4} we will equate \eqref{jepsilon} to $j$ in order to translate formulas given in terms of $\e$ in Section~\ref{A} to similar formulas for $E$.
Up to a simple transformation, $E$ is the universal elliptic curve $E_j$ from Diamond and Shurman's book~\cite{DandS}.
Sections 7.5 and 7.7 in \cite{DandS} show that the modular curve $X_1(N)$ can be represented with the equation $q_N$ when $N > 2$. 
In particular, $q_N$ is irreducible in $\QQ[x, j_0]$. Likewise $F_N$ is irreducible in $\ZZ[F_3,F_4]$.
Although $q_2 \not\in \ZZ[x,j_0]$, its square $4 (x^3 - 3 j_0 x - 2 j_0)$ is an equation for $X_1(2)$ that lies in $\ZZ[x,j_0]$.

\section{The valuation of a division polynomial at a cusp} \label{4}

Recall from Section~\ref{2.1} that a {\em place} on $X_1(N)/\QQ$ is a discrete valuation $v_P: \QQ(X_1(N)) \rightarrow \ZZ \bigcup \{\infty\}$. Such a place is a {\em cusp over}\, $\QQ$ when $v_P( j ) < 0$.
A function $g \in \QQ(X_1(N))$ is called a {\em modular unit} if every place with $v_P(g) \neq 0$ is a cusp.
If $k\neq N$ and $k > 2$, then $F_k$ is a {\em modular unit} in $X_1(N)$, see \cite[Section 2]{DerickxvanHoeij}.
However, to obtain a modular unit from $q_2$, it was necessary to take its $4^{\text{th}}$ power and scale it to
\begin{equation} \label{F2}
	F_2 = \frac{q_2^4}{ 1728 j_0^2 (j_0-1) }.
\end{equation}


Let $s = 1/j$, then {\em a cusp over $\QQ$} is a place above $s=0$, which corresponds to a {\em conjugacy class of Puiseux expansions at $s=0$}, see Section~\ref{2.1}.
Conjugation is always over $\QQ((s))$ in this paper.

From~\eqref{RedOdd},\eqref{RecEven}
one can compute 
$Q_2,Q_3,\ldots$ and then $q_2^2,q_3,q_4,\ldots \in \ZZ[x,j_0] \subset \QQ(s)[x]$.
We computed Puiseux expansions of $q_N$ (or $q_2^2$ if $N = 2$) at $s=0$ for $N \leq 9$.
Newton's algorithm gives arbitrarily many terms,
but only dominant terms will be needed. 
Table \ref{PuiseuxExpansionsCusps} lists the {dominant term} of $\bfp+1$ for {\em one}\, Puiseux expansion $\bfp$ {\em from each}\, conjugacy class $\{\bfp, \ldots \}$ (which has $\bfw_{\bfp} := {\bf e}_{\bfp}\,{\bf f}_{\bfp}$ elements, see Section~\ref{2.1}).

We use $\bfp +1$ and $x+1$ rather than $\bfp$ and $x$ because when $j \to \infty$ the curve $E$ in~\eqref{E} becomes singular at $x=-1$.
This is in contrast to $E_{\epsilon}$
which becomes singular at $x=0$ when $\e \to 0$.

\begin{table}[h!]
\centering
\begin{tabular}{|r|r|r|r|r|r|}   
\hline
        & $l_s(\bfp+1)$                 & $l_s(\bfp+1)$   & $l_s(\bfp+1)$   & $l_s(\bfp+1)$ & $l_s(\bfp+1)$ \\
        &  with $\bfp \in C_0$       & $\bfp \in C_1$  & $\bfp \in C_2$   & $\bfp \in C_3$ & $\bfp \in C_4$  \\ \hline
$q_2^2$ & 3 & $-24 s^{1/2}$ &&& \\ \hline
$q_3$ & 4 & $-12s^{1/3}$ &&& \\ \hline
$q_4$ & 6 & $-12s^{1/4}$ & ${\color{gray}0 \,s^{1/2}}-672 s$ && \\ \hline
$q_5$ & $3 \sin(\pi/5)^{-2}$ & $-12 s^{1/5}$ &  $-12s^{2/5}$ && \\ \hline
$q_6$ & 12 & $-12s^{1/6}$ & $-12(-s)^{1/3}$ &  $-12s^{1/2}$& \\ \hline
$q_7$ & $3 \sin(\pi/7)^{-2}$ & $-12s^{1/7}$ & $-12s^{2/7}$ & $-12s^{3/7}$& \\ \hline
$q_8$ & $3 \sin(\pi/8)^{-2}$ & $-12s^{1/8}$ & $-12 (-s)^{1/4}$ & $-12s^{3/8}$ &$-12 (2s)^{1/2}$  \\ \hline
$q_9$ & $3 \sin(\pi/9)^{-2}$ & $-12s^{1/9}$ & $-12s^{2/9}$      &  $-12(\zeta_3 \cdot s)^{1/3}$ & $-12s^{4/9}$  \\ \hline
\end{tabular}
\vspace{.05 in}
\caption{$l_s(\bfp+1)$ for one $\bfp$ from each conjugacy class $C_i$ over $\QQ((s))$} 
\label{PuiseuxExpansionsCusps}
\end{table}
The equation for $X_1(2)$ is $q_2^2 = 4(x^3 - 3 j_0 x  - 2 j_0)$, where $j_0 = j/(j-1728) = 1/(1-1728 s)$.
To illustrate Table~\ref{PuiseuxExpansionsCusps} for $N=2$,
factor $q_2^2=4(x-\bfp_0)(x-\bfp_{1a})(x-\bfp_{1b}) \in \overline{ \QQ((s)) }[x]$. Row $q_2^2$ in Table~\ref{PuiseuxExpansionsCusps} gives $l_s(\bfp_0 + 1) = 3$,
$l_s(\bfp_{1a} +1 ) = -24s^{1/2}$, and its conjugate $l_s(\bfp_{1b} +1 ) = 24 s^{1/2}$. This means
\begin{equation} \label{x3}
	q_2^2 = 4((x+1)-3+\cdots )\left((x+1) +24 s^{1/2} + \cdots \right)\left((x+1) - 24s^{1/2} + \cdots \right),
\end{equation}
where the dots indicate terms with higher powers of $s$.
Likewise, for $N>2$,
\[ q_N = a_N \prod_{c=0}^{\lfloor N/2 \rfloor} \left(x - \bfp_{c,*}\right), \]
where $l_s\left(\bfp_{c,*}+1\right)$ are the conjugates of the term listed in row $q_N$, column $C_c$.
\begin{example}\label{exampleq8} Counting conjugates, row $q_8$ in Table~\ref{PuiseuxExpansionsCusps} gives two $\bfp$'s with $v_s(\bfp+1) = 0$, eight $\bfp$'s with $v_s(\bfp+1) = 1/8$, four with $1/4$, eight with $3/8$, and two with $1/2$.
Indeed, ${\rm deg}_x(q_k) = m_8/2$ equals $2+8+4+8+2$.

Now take as an example the conjugacy class $C_1$ (a {\em cusp over $\QQ$}) on $X_1(3)$. Row $q_3$, column $C_1$ gives $\bfp + 1 =  -12s^{1/3} + \cdots$.
Viewing $q_8$ as an element of $\QQ\left(X_1(3)\right) = \QQ(s)[x]/\hspace{-.1 cm}\left(q_3\right)$, we can insert this data into Equation~(\ref{lucky}) to find
\begin{equation} \label{example_8_3}
	v_P\left(q_8\right)\hspace{-.02 in} = 3 \cdot \hspace{-.02 in} \left(2\,{\rm min}\left(\frac13,0\right)\hspace{-.02 in} + 8\,{\rm min}\left(\frac13,\frac18\right)\hspace{-.02 in} + 4\,{\rm min}\left(\frac13,\frac14\right)\hspace{-.02 in} + 8\,{\rm min}\left(\frac13,\frac38\right)\hspace{-.02 in} + 2\,{\rm min}\left(\frac13,\frac12\right)\hspace{-.02 in}\right)
\end{equation}
(the $+1$'s cancelled out). Omitting the factor of 3 gives the {\em unweighted order} from Section~\ref{2.1}.
\end{example}
If $1<N$, $k \leq 9$, and $N \neq k$, then Example \ref{exampleq8} shows how one can use Table \ref{PuiseuxExpansionsCusps} to compute the valuation of $q_k$ (or $q_2^2$ if $k=2$) at any cusp of $X_1(N)$. To find a general formula, we will show that Observations \eqref{Observation1}--\eqref{Observation6} below, which hold in Table~\ref{PuiseuxExpansionsCusps}, hold for all $N>1$.

\begin{enumerate}
\item\label{Observation1} $q_N$ ($q_2^2$ if $N=2$) has $ \lfloor N/2 \rfloor+1$ conjugacy classes (a.k.a. Galois orbits) $C_0,C_1,\ldots,C_{\lfloor N/2\rfloor}$ of Puiseux expansions at $s=0$.  
We number them so that if $\bfp \in C_c$ then $v_s(\bfp+1) = c/N$ except when $(N,c)=(4,2)$. This unique exceptional case is the irregular cusp of $X_1(4)$, where the $s^{c/N}$ term in Table \ref{PuiseuxExpansionsCusps} is $0 s^{1/2}$
and $v_s(\bfp+1)=1$ instead. 

\item\label{Observation2} $C_0$ has $l_s(\bfp+1) = 12/\hspace{-2pt}\left(2-\zeta_N - \zeta_N^{-1}\right) 
= 3 \sin(\pi/N)^{-2}$.
The residue field is $\QQ\left( \zeta_N + \zeta_N^{-1}\right)$.

\item\label{Observation3} If $0<c<N/2$, then $\bfp \in C_c$ has $l_s(\bfp+1) = -12\left(\zeta_d \cdot s\right)^{c/N}$ (always up to conjugation) with $d = \gcd(c,N)$ and residue field $\QQ(\zeta_d)$. 

\item\label{Observation4} If $N>4$ is even, then $\bfp\in C_{N/2}$ 
has \[l_s(\bfp+1) = -24s^{1/2} + 3 \sin^2( \pi / N) \cdot 16 s^{1/2} = -12 (\beta \cdot s)^{1/2},\]
where $\beta := \left(\zeta_N + \zeta_N^{-1}\right)^2$.
The residue field is $\QQ(\beta)$ (recall Example~\ref{why}).   

\item\label{Observation5}
$C_c \subset \CC\LPar s^{1/{\bf e}_c(N)}\RPar$ has precisely $\bfw_c(N)$ elements, and the residue field has degree ${\bf f}_c(N)$  with $\bfw_c, \, {\bf e}_c, \, {\bf f}_c$ as in Section~\ref{A}.

\item\label{Observation6} Every $\bfp \in \bigcup_{N,c} C_c(N)$ has a unique $l_s(\bfp+1)$,
so Equation~(\ref{lucky}) holds for all combinations. 
This implies that Example~\ref{exampleq8} generalizes to Theorem~\ref{Main} below.

\end{enumerate}

To see why Observations \eqref{Observation1}--\eqref{Observation6} 
hold, note that the curve $E_{\e}$ in Section~\ref{A} differs from $E$ by the transformation 
\[T: x \mapsto \bfp_{1a} + (\bfp_{0}-\bfp_{1a}) x = (-1-24 s^{1/2}+\cdots) + (3+24 s^{1/2}+\cdots)x
\] 
%
that sends $0,\e,$ and $1$ to $\bfp_{1a}, \bfp_{1b},$ and $ \bfp_0$, respectively.
From $T(\e) = \bfp_{1b}$ we find $\e = 16 s^{1/2} + \cdots$
which can also be computed by equating $j = 1/s$ to~\eqref{jepsilon}.
Section~\ref{A} gives the $\e$-dominant terms. 
Substituting $\e \sim 16 s^{1/2}$ and applying $T$ 
yields $l_s(\bfp + 1)$ for every
Puiseux expansion 
of $q_N$. 
Observation~\eqref{Observation6} immediately follows from this, but then Lemma~\ref{l} shows that
${\bf e}_{\bfp}$ and ${k}_{\bfp}$ can be read from $l_s(\bfp + 1)$, and the remaining observations follow.



\begin{theorem} \label{Main} {\bf (MinFormula)}. For $t \in [0,1/2]$, we define the following {\em unweighted order functions}. For $k=2$, define $v_2(t) =  4t-1$,  
and for $k>2$, define 
%
\begin{equation} \label{vk}
	v_k(t)  = s_k \cdot  \left( -\frac{m_k}3 t + \sum_{c=1}^{\lfloor k/2 \rfloor}  \bfw_c(k)  \min\left(t, \frac{c}k\right) \right),
\end{equation}
where $s_3=3$ and $s_k=1$ for $k>3$.  Recall that $C_c(N)$ is a conjugacy class of Puiseux expansions, giving one cusp of $X_1(N)/\QQ$,
or a Galois orbit with ${\bf f}_c(N)$ cusps of $X_1(N)/\overline{\QQ}$.

Let $2 < N \neq k >1$ and  $0 \leq c \leq N/2$, 
then $F_k$, viewed as element of $\QQ(s)[x]/(q_N) = \QQ(X_1(N))$, has order ${\bf e}_c(N) \cdot v_k\left(\frac{c}N\right)$ at $C_c(N)$.
\end{theorem}
If $N=2$ we can not directly apply this formula to $F_k$ due to its denominator $q_2$, 
but the formula still holds for products where $q_2$ cancels out, such as $F_2^2 F_3$ and $F_2^{m_k/12} F_k$ for $k>3$.
\begin{proof} Observations \eqref{Observation1}--\eqref{Observation6} imply that the computation in Example~\ref{exampleq8} works in general, so
\begin{equation} \label{valq}
	{\bf e}_c(N) \sum_{j=0}^{\lfloor k/2 \rfloor} \bfw_{j}(k) \min\left(\frac{c}N, \frac{j}k \right)
\end{equation}
is the order of $q_k$ (or $q_2^2$ if $k=2$)
at $C_c(N)$ for any $N,k > 1$ with $N \neq k$. 
Theorem \ref{Main} follows by applying Equation~(\ref{valq}) to $F_k$ in Equations~(\ref{F3}) and (\ref{F2}), simplifying $\min(t,0)=0$ and $\min(t,1/2)=t$,
and noting that the denominator $1728 j_0^2 (j_0 - 1)$ in Equation (\ref{F2}) has a root of order 1 at $s=0$.
\end{proof}

\begin{remark} \label{rqt}
A cusp over $\QQ$ corresponds to ${\bf f}_c(N)$ cusps over $\overline{\QQ}$.
Since the degree of the divisor of a function is zero, 
\[\sum_{c=0}^{\lfloor N/2\rfloor} {\bf f}_c(N)\,{\bf e}_c(N) \,  v_k\left(\frac{c}N\right)=0.\]
%
If $N$ is prime, then ${\bf e}_c(N) \, {\bf f}_c(N) = \bfw_c(N)$ is $N$ for $c>0$, and $v_k(0)=0$, so
\[\sum_{c=0}^{\lfloor N/2\rfloor}N \, v_k\left(\frac{c}N\right)=0.\] 
Letting $N \rightarrow \infty$, we see $\int_{t=0}^{1/2} v_k(t) {\rm d}t = 0$.
Since $\int_{t=0}^{1/2} \left( \min(t, \frac{c}k) - 4 \frac{c}k(1- \frac{c}k) t\right) {\rm d}t$ equals $0$ for any $\frac{c}k \in [0,1/2]$, we do not need a formula for $m_k$, and can instead rewrite Equation~(\ref{vk}) as
\begin{equation}
	\label{valqt}
	v_k(t) = s_k \sum_{0<c<k/2}  \bfw_c(k)  \left( \min\left(t, \frac{c}k\right) - 4 \frac{c}k \left(1-\frac{c}k\right) t \right), \ \ \ \ {\rm for} \ k \geq 3.
\end{equation}
This is the formula implemented in \cite{MinFormula}. The sum~(\ref{valqt}) does not change if one replaces $0<c<k/2$ by $0\leq c \leq k/2$ because the summand vanishes at $c \in \{0, k/2\}$.
Equation~(\ref{valqt}) with the factor $s_k$ removed gives the unweighted order function for $\tilde{q}_k$ (recall that $F_3 = \tilde{q}^3$ and $F_k = \tilde{q}_k$ if $k>3$).
\end{remark}

\section{The degree of $F_7/F_8$ in $X_1(N)$} \label{5}

In this section we use Theorem~\ref{Main} to prove an upper bound for the $\QQ$-gonality of $X_1(N)$.

\begin{figure}[h!]
\centering
\begin{tikzpicture}[scale = 1]
\begin{axis}[ 
	xmin=0, xmax=.54,
	ymin=0, ymax=.17,
	xtick={0,0.1,...,0.5}, ytick={0,0.05,...,.15}, 
	xticklabels={0, , , ,$\frac{2}{5}$,$\frac{1}{2}$},
	yticklabels={0, , , },
	major tick length={0},
	grid=major,
	line width=1pt, 
	axis lines=center
	] 
	\path[name path=axis] (axis cs:0.25,0) -- (axis cs:0.333333333,0);
	\addplot [thick, name path = p1, domain=1/4:2/7] {4*x-1}; 
	\addplot [
        thick,
        color=Gray,
        fill=Gray, 
        fill opacity=1.0
    ]
	fill between[
        of=p1 and axis,
        soft clip={domain=0:1},
    ];
	\addplot [thick, domain=2/7:1/3] {1-3*x};
	\addplot [thick, name path = p2, domain=2/5:3/7] {5*x-2}; 
	\path[name path=axis2] (axis cs:0.4,0) -- (axis cs:0.5,0);
	\addplot [
        thick,
        color=Gray,
        fill=Gray, 
        fill opacity=1.0
    ]
	fill between[
        of=p2 and axis2,
        soft clip={domain=0:1},
    ];
	\addplot [thick, domain=3/7:1/2] {1-2*x}; 
\end{axis}
\node [below] at (3.18,0) {$\frac{1}{4}$};
\node [below] at (4.25,0) {$\frac{1}{3}$};
\node [left] at (3.65,4.8) {$\left(\frac{2}{7},\frac{1}{7}\right)$};
\node [right] at (5.43,4.8) {$\left(\frac{3}{7},\frac{1}{7}\right)$};
\node [below] at (6.37,0) {$\frac{1}{2}$};
\end{tikzpicture}
\caption{The function $m(t)$ graphed from $0$ to $\frac{1}{2}$}
\protect{\label{mgraph}}
\end{figure}
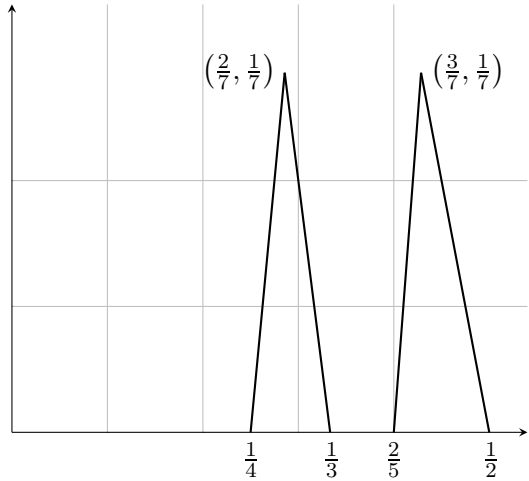

Let $v(t) := v_7(t) - v_8(t)$, and let $m(t) := {\rm max}(0, v(t))$ as in Figure~\ref{mgraph}. 
Define
\begin{equation*}\label{notationref}
	B_0(N) := \sum_{0 < c < N/2}\bfw_c(N) \, m\left(\frac{c}N\right) \  \text{ and } \ B_1(N) := \sum_{0 < c < N/2} N\,m\left( \frac{c}N\right).
\end{equation*}
Theorem~\ref{Main} gives
\begin{equation} \label{ineq}
 {\rm div}\hspace{-.03 in}\left(\frac{F_7}{F_8}\right) = \sum_{0 < c < N/2} {\bf e}_c(N)\,v\left(\frac{c}N\right)\,C_c, {\rm \ \ \ \ so \ \ }
{\rm deg}\hspace{-.03 in}\left(\frac{F_7}{F_8}\right) = B_0(N) \leq B_1(N).
\end{equation}
We omit the terms $c=0$ and $c=N/2$ in these sums because $v$ vanishes there.  By Equation~(\ref{vk})
\[ v(t) = 7 \min\left(t,\frac17\right)+7 \min\left(t,\frac27\right)+7 \min\left(t,\frac37\right)
-8 \min\left(t,\frac18\right) - 4 \min\left(t,\frac14\right) - 8 \min\left(t,\frac38\right) - 2t.\]

\begin{lemma}\label{420Cases} 
If $N$ is relatively prime to $420 =3\cdot 4\cdot 5\cdot 7$, then $B_1(N) = [ 11 N^2/840 ]$. In general, $B_1(N) \leq  [ 11 N^2/840 ]+2$ (where equality implies $7\mid N$), and $B_0(N)\leq  [ 11 N^2/840 ]$.
\end{lemma}

\begin{proof}\label{Defmi} Consider the intervals $I_1 := (1/4, 2/7]$, $I_2 := (2/7, 1/3)$, $I_3 := (2/5, 3/7]$, and $I_4 := (3/7, 1/2)$. These intervals partition the support of $m(t)$; see Figure \ref{mgraph}.
We define the functions $m_1(t) := 4t-1$, $m_2(t) := 1-3t$, $m_3(t) := 5t-2$, and $m_4(t):=1-2t$. The graphs of the $m_j(t)$ over $I_j$ are exactly the line segments in Figure \ref{mgraph}.
We see $m(t) = m_j(t)$ if $t \in I_j$ and 0 otherwise. 

Our goal is to bound
\begin{equation} \label{B1} B_1(N) 
  = \sum_{j=1}^4 \sum_{\substack{ \ c/N \in I_j \\  \ c\in \ZZ}}  N\,m_j\hspace{-.08 cm}\left(\frac{c}N\right). \end{equation}
Since $N\,m_j(c/N) \in \ZZ$, we have $B_1(N)\in \ZZ$. Note that $B_1(N)$ is a Riemann sum of 
\[
	N^2 \int_{t=0}^{1/2} m(t) {\rm d}t = N^2 \sum_{j=1}^4 \int_{I_j} m_j(t) {\rm d}t = \frac{11}{840} N^2.
\]

Since $m(t)$ is piece-wise linear, any error in $B_1(N)$, viewed as an approximation to this integral, must come from the corners:  
\[ \frac{1}{4}, \  \frac{2}{7}, \  \frac{1}{3}, \  \frac{2}{5}, \  \frac{3}{7}, \text{ and }  \frac{1}{2}.\] 
This error depends only on $N$ modulo $420=4\cdot 7\cdot 3\cdot 5.$
To demonstrate this, let $c_{1a}$ and $c_{1b}$, respectively, be the minimum and maximum integer $c$ with $c/N \in I_1$.
Then $c_{1a}$ equals 
\[\frac{N+4}{4}, \frac{N+3}{4}, \frac{N+2}{4}, \text{ or } \frac{N+1}{4}\] 
depending on whether $N$ is, respectively, 0, 1, 2, or $3 \bmod 4.$
Likewise, the expression for $c_{1b}$ in terms of $N$ depends only on $N \bmod 7.$
Considering the intervals $I_2$, $I_3$, and $I_4$, we have a total of $4 \cdot 7 \cdot 3 \cdot 5 = 420$ cases. Hence, the difference between the integral and its Riemann sum $B_1(N)$ depends only on $N \bmod 420.$ 

As example we cover one of the 420 cases, namely $N \equiv 32 \bmod 420.$ Here $c_{1a} = (N+4)/4$
and $c_{1b} = (2N-1)/7$, so 
\[\sum_{\substack{c/N \in I_1 \\ c\in\ZZ}} N\,m_1\hspace{-.08 cm}\left(\frac{c}{N}\right)\]
has $n := c_{1b}-c_{1a}+1 = N/28-1/7$ terms. The average of these $n$ terms is
\[\frac12 {\left( N\,m_1\hspace{-.08 cm}\left(\frac{c_{1a}}{N}\right) + N\,m_1\hspace{-.08 cm}\left(\frac{c_{1b}}{N}\right) \right)} = \frac{15N}{14}+\frac{5}{7},\]
so the $j=1$ part of $B_1(N)$ in Equation~(\ref{B1}) is $(N/28-1/7) \cdot (15N/14+5/7)$.
Repeating this computation for $j=2,3,4$ and summing, we find 
\[B_1(N) = \frac{11 N^2 }{ 840} + \frac{43}{105}.\] 
Since $|43/105| < 1/2$ we have
$B_0(N)\leq B_1(N) = [11 N^2/840]$ for any $N \equiv 32 \bmod 420$.

In the same way we calculated the difference between $B_1(N)$ and $\left[11N^2/840\right]$ for all 420 cases, the programs are available at \cite{SageComputations}. 
In all cases with $\gcd(N,420)=1$, we found $B_1(N)=\left[11N^2/840\right]$.

If $N$ is prime and $c>0$, then the
factor
$\bfw_c(N)$ in the definition of $B_0(N)$ is $N$, and hence $B_0(N) = B_1(N)$.
So for primes $N>7$ we find
\begin{equation} \label{pr} \deg\left(F_7/F_8\right)=B_0(N)=B_1(N)=\left[11N^2/840\right].\end{equation}

For cases with $\gcd(N,7)=1\neq \gcd(N,2\cdot 3\cdot 5)$ the computation found $B_1(N)\leq \left[11N^2/840\right]+1$. Moreover,
in these cases there is a $c/N$ in some $I_j$ with  $\bfw_c(N) < N$. Then $B_0(N) < B_1(N)$, and so $B_0(N)\leq \left[11N^2/840\right]$.

For the remaining cases $\gcd(N,7)\neq 1$ we found $B_1(N) \leq \left[11N^2/840\right]+2$. The smallest $N$ for which that is sharp is $N=49$.
We check that a multiple of $7/N$ is in each of the intervals $(1/4,1/3)$ and $(2/5,1/2)$. These two multiples $c/N$ of $7/N$ each have $\bfw_c(N) < N$.
%
%
So $B_0(N) \leq B_1(N) - 2$.
Hence we still have $B_0(49) \leq \left[11\cdot 49^2/840\right]$.  
The next 
$N$ with $B_1(N) = \left[11N^2/840\right]+2$ is $N=91$. For $N \geq 91$ the intervals $(1/4,1/3)$ and $(2/5,1/2)$ each have at least 7 consecutive $c/N$'s, and so ${\rm gcd}(c,N) > 1$ (which implies $\bfw_c(N) < N$)
happens at least once in each of those intervals.
Then the same argument 
shows that $B_0(N) \leq \left[11N^2/840\right]$.
\end{proof}


From Lemma \ref{420Cases} we obtain:

\begin{theorem}\label{mainthm}
For $N \neq 7, 8$ the modular unit 
\[\frac{F_7}{F_8}:X_1(N)\to \PP^1\] 
has degree 
\[{\rm deg}\hspace{-.08 cm}\left(\frac{F_7}{F_8}\right) = B_0(N) 
\leq \left[ \frac{11 N^2}{840}\right],\] 
with equality when $N$ is prime.
If $N>8$, then 
this is an upper bound for the gonality.
\end{theorem}

\begin{proof} We need $N \neq 7,8$ to ensure $F_7,F_8 \neq 0$.
If $N<7$ then $B_0(N)=0$ which means that $F_7/F_8$ is constant.
The degree of a {\em non-constant} function is an upper bound
for the gonality.
It is easy to check that $B_0(N) > 0$ for $N > 8$.
\end{proof}

If $N$ is not prime, then the gonality is usually smaller than $B_0(N)$, see~\cite[Table~1]{DerickxvanHoeij}.
If $N>8$ is prime, then equation~\eqref{pr} 
gives an excellent gonality bound; the only primes $N <250$ for which a sharper bound is known are 31, 67, 101 and for these cases, that bound is only one less. 


\section{Second proof for MinFormula} \label{6}

\subsection{Cusps:  A modular interpretation} \label{Section6.1}
Take the congruence subgroup \[\Gamma_1(N)=\left\{\begin{bmatrix}
  a & b\\
  c & d
\end{bmatrix}\in \SL_2(\ZZ) \ \middle| \ \begin{bmatrix}
  a & b\\
  c & d
\end{bmatrix}\equiv 
\begin{bmatrix}
  1 & *\\
  0 & 1
\end{bmatrix} \bmod  N \right\}\]
where $*$ indicates the entry is unspecified. 
The extended complex upper half plane is
\[\overline{\Hcal}=\Hcal \cup \QQ \cup \{\infty\},\] 
where $\Hcal$ is the usual complex upper half plane.
The groups $\Gamma_1(N) \subseteq \SL_2(\ZZ)$ act
on the extended complex upper half plane $\overline{\Hcal}$ by fractional linear transformations.
The quotient is the modular curve $X_1(N)$.

%
Following \cite[Chapter 3.8]{DandS} and similar to Section~\ref{A}, we 
represent cusps of $X_1(N)/\hspace{.05 cm}\overline{\QQ}$  with pairs of order $N$ vectors
\[\pm \begin{bmatrix}
           a \\
           c \\
         \end{bmatrix}\in \left(\ZZ/N\ZZ\right)^2.\]
%
The Galois action on the cusps can be represented with matrices of the form
\[\pm\begin{bmatrix}
  y & z\\
  0 & 1
\end{bmatrix}\in \GL_2(\ZZ/N\ZZ)\]
on the order $N$ vectors in $\left(\ZZ/N\ZZ\right)^2$, see \cite[Sections 7.6-7.7]{DandS}.
Two vectors
         \[\begin{bmatrix}
           a' \\
           c' \\
         \end{bmatrix} \ \ \ \text{ and } \ \ \ \begin{bmatrix}
           a \\
           c \\
         \end{bmatrix}\] 
represent the same cusp when
         \[\begin{bmatrix}
           a' \\
           c' \\
         \end{bmatrix}  = \pm \begin{bmatrix}
           a + jc \\
           c \\
         \end{bmatrix}\] 
for some $j \in \ZZ$.
Two cusps represented this way are in the same Galois orbit if and only if $c = \pm c'$. So each Galois orbit is uniquely determined by $\pm c$, in other words,
by an element of the \emph{Cartan}\,  $C(N):=\left(\ZZ / N\ZZ\right)/\pm$, which is identified with $\{0,\ldots,\lfloor N/2 \rfloor\}$. 
We will denote such orbit by $C_c(N)$. 
%
%
Let $\bfn_c(N), \bfe_c(N), \bff_c(N)$ be as in Section~\ref{A}.
There are $\bff_c(N)$ cusps in $C_c(N)$, each of which is represented by $\bfe_c$ pairs of vectors in $(\ZZ/N\ZZ)^2$, for a total of $\bfn_c = \bfe_c\,\bff_c$ pairs.

%
%
%

The \emph{width} 
of a cusp (\cite[pages 59 and 60]{DandS}) is defined as follows.
Let $A\in \SL_2(\ZZ)$ be such that $A\cdot\left[\begin{smallmatrix} a\\ c \end{smallmatrix}\right]=\infty=\left[\begin{smallmatrix} 1\\ 0 \end{smallmatrix}\right]$. The \emph{width} ${\bf e}_{\scriptscriptstyle{\left[ {a \atop c} \right]}}(N)$
is the smallest positive integer for which 
\begin{equation*}\label{WidthNotRef}
A\begin{bmatrix}
  1 & {\bf e}_{\scriptscriptstyle{\left[ {a \atop c} \right]}}(N)\\
  0 & 1
\end{bmatrix}A^{-1}\in \Gamma_1(N).
\end{equation*} 
A computation 
shows  that this is ${N}/{\gcd(c,N)}$.
So the \emph{width} ${\bf e}_{{\scriptscriptstyle{\left[ {a \atop c} \right]}}}(N)$ is ${N}/{\gcd(c,N)}$, which equals the number $\bfe_c(N)$ from Sections~\ref{A} and~\ref{4}
with {\em one exception}, namely $C_2(4)$.
The cusp corresponding to $\left[\begin{smallmatrix} 1\\ 2 \end{smallmatrix}\right]$ on $X_1(4)$ is the lone cusp in the orbit $C_2(4)$.
It is the only irregular cusp for any modular curve $X_1(N)$, $X_0(N)$, or $X(N)$ (\cite[page 75]{DandS}). It has width $2$, but it has `order' 1.
Throughout this paper $\bfe_c(N)$ denotes the width, except for the case $\bfe_2(4)$ where it denotes the `order' 1.



\subsection{Siegel Functions}\label{SiegelFunctions}

We would like to define a class of functions on the complex upper half plane $\Hcal$. 


\begin{definition}\label{defSiegel}
Let $(a_1,a_2) \in \QQ^2 - \ZZ^2$. For $\tau \in \Hcal$, define the \emph{Siegel function} associated to $(a_1,a_2)$, denoted $g_{(a_1,a_2)}$, by the product
\begin{align*}&g_{(a_1,a_2)}(\tau):=-q^{\frac{1}{2}\BB_2(a_1)}e^{2\pi i\frac{1}{2}(a_2(a_1-1))}(1-e^{2\pi ia_2}q^{a_1})\prod_{n=1}^\infty (1-e^{2\pi ia_2}q^{n+a_1})(1-e^{-2\pi i a_2}q^{n-a_1}),
\end{align*}
where $q=e^{2\pi i \tau}$, and $\BB_2(x)=x^2-x+\frac{1}{6}$\label{SecondBernoulli}
is the second Bernoulli polynomial.
\end{definition}
One can check that adding an integral vector to $(a_1,a_2)$ does not change the order of $g_{(a_1,a_2)}$,
so we can interpret $(a_1,a_2)$ as a non-zero element of  $\left(\QQ / \ZZ \right)^2$.

We are interested in the divisors of Siegel functions. From the $q$-expansion, we see that
\[\ord_\infty g_{(a_1,a_2)}= {\bf e}_\infty \cdot \frac{1}{2}\,\BB_2(a_1).\]
Recall, $\infty$ denotes the standard prime at infinity given by the equivalence class of 
$\left[\begin{smallmatrix} 1\\ 0 \end{smallmatrix}\right]$ under the action of $\Gamma_1(N)$ and ${\bf e}_\infty = 1$ is its width.
Consider another cusp of the modular curve $X_1(N)$ that corresponds to the orbit of $\left[\begin{smallmatrix} a\\ c \end{smallmatrix}\right]$. Let $A\in \SL_2(\ZZ)$ be a matrix such that 
\[A\cdot\begin{bmatrix} 1 \\ 0 \\ \end{bmatrix}=\begin{bmatrix} a \\ c \\ \end{bmatrix}.\] When $g_{(a_1,a_2)}$ is a function on $X_1(N)$, the \emph{order of $g_{(a_1,a_2)}$} at the cusp corresponding to $\left[\begin{smallmatrix} a\\ c \end{smallmatrix}\right]$  is
\begin{equation}\label{ordg}
	\ord_{{\scriptscriptstyle{\left[ {a \atop c} \right]}}}\left(g_{(a_1,a_2)}\right) 
	= 
	{\bf e}_c
	\cdot \frac{1}{2} \,\BB_2\left(\,\left\{ \, \left[(a_1,a_2)\cdot A\right]_1  \, \right\}\,\right),
\end{equation}
where $\{\bullet\}=\bullet-\lfloor\bullet\rfloor$ denotes the {\em fractional part}\, and $[\bullet]_1$ denotes the first entry of the vector. The paper \cite{SutherlandZywina} has a concise description of the above for an arbitrary modular curve, but \cite[Chapter 2]{KubertLang} has a more thorough exposition for $X(N)$; specifically, see the boxed equation on page 40. The reader should note that in \cite{KubertLang}, Kubert and Lang are considering the $q^\frac{1}{N}$ expansion. 
%
%
In the remainder of this paper, we will consider Siegel functions of the form $g_{(0,a)}$, with $a$ a nonzero element of $\QQ/\ZZ$ of order dividing $N$.
Following \cite{Streng}, we write
\[H_{k} := g_{\left(0,\frac{k}N\right)}, \ \ \ {\rm with} \ \ k \in \ZZ - N \ZZ. \]
\textbf{Caution:  }  In \cite{Streng}, Streng considers the modular curve $X^1(N)$, while we have $X_1(N)$.
The isomorphism $\Gamma^1(N)\setminus\Hcal \to \Gamma_1(N)\setminus \Hcal$ is given by 
$\begin{bmatrix}
  0 & -1\\
  1 & 0
\end{bmatrix}.$ This isomorphism sends $g_{(a,0)}$ to $g_{(0,-a)}$; however, $g_{(0,-a)}=-g_{(0,a)}$.

\vspace{.3 cm}

The \emph{unweighted order of $H_k$} at $c\in C(N)$ is 
\begin{equation} \label{Ord}
	{\rm uord}_c\left(H_k\right) = \frac{1}{2}\,\BB_2\left(\left\{c\cdot\frac{k}N\right\}\right).
\end{equation}
%
%
%
Note that $x \mapsto \BB_2(\{x\})$ is a continuous function even though $x \mapsto \{x\}$ is not.

\subsection{Generators of the Modular Units} \label{6.3}
Describing the modular units on a given modular curve has long been a subject of interest. A significant motivation of Kubert and Lang's text \cite{KubertLang} is to describe the units of $X(N)$ over $\QQ(\zeta_N)$. They show that, with the exception of some 2-torsion elements when $N$ is even, the units are generated by the Siegel functions described above.

Motivated by~\cite[Conjecture 1]{DerickxvanHoeij}, Streng \cite{Streng} has used similar methods 
to describe all modular units on $X^1(N)$ over $\QQ$. Before stating the result we introduce some of the relevant objects. 
We start with Tate normal form. 

\begin{lemma}\label{Tatenormalform} (\cite[Lemma 2.1]{Streng}). If $E$ is an elliptic curve over a field $K$ of characteristic 0 (such as the elliptic curves in Equations (\ref{defE}) and (\ref{E}))
and $P$ is a point on $E$ of order greater than 3 with $x(P) \in K$, then the pair $(E,\pm P)$
is isomorphic to a unique pair of the form
\begin{equation}\label{Tatenormal}
E_T:Y^2+(1-C)XY-BY=X^3-BX^2, \ \ \ \ \  P=(0,0),
\end{equation} 
where $B,C\in K$ and the discriminant 
\[  D=B^3(16B^2+(1-20C-8C^2)B+C(C-1)^3)\neq0.\]
Further, each pair $B,C\in K$ with $D\neq0$ satisfying \eqref{Tatenormal} yields an elliptic curve and with a distinguished point $P$ of order greater than 3.
\end{lemma}
This form $E_T$ is called \emph{Tate normal form}.  Let $K = \QQ(j)$ and $E$ be as in Equation $(\ref{E})$ and let $K_0 = K(x_0)$ where $x_0$ is transcendental over $K$.
Let 
\[P = \left(x_0, \sqrt{x_0^3 - 3 j_0 x_0 - 2 j_0}\right).\]
Sending $P$ to $(0,0)$ and $E$ to Tate normal form with affine linear transformations results
in expressions $B, C \in K_0$ (computation at \cite{SageComputations}). Identify $x_0$ with $x$ so that $K_0$ becomes $\QQ(x,j)$.
Then $B,C \in \QQ(x,j)$ are $B = -F_3$ and $C = -F_4$.
Due to the uniqueness of the Tate normal form, it should also be possible to write $x,j$ in terms of $B,C$, and a computation \cite{SageComputations} confirms that.
Thus $\QQ(B,C) = \QQ(x,j)$. 

A computation \cite{SageComputations} shows $F_2 = B^4/D$.
Conjecture 1 in \cite{DerickxvanHoeij}, proved by Streng \cite{Streng},
says that for $N > 2$, the modular units in $\QQ(X_1(N))$ modulo $\QQ^*$ are freely generated by
$F_2,\ldots,F_{\lfloor N/2 \rfloor + 1}$.

Considering the Tate normal form over $\ZZ[B,C]$,
we can look at the $k^{\text{th}}$ division polynomial $\psi_{k,E_T}(x,y)\in \ZZ[B,C][x,y]$. As in \cite[Example 2.2]{Streng}, evaluating $\psi_{k,E_T}$ at $(0,0)$ gives:
\begin{align*}
&P_1:=\psi_{1,E_T}(0,0)=1,  \quad \quad P_2:=\psi_{2,E_T}(0,0)=-B, \quad \quad P_3:=\psi_{3,E_T}(0,0)=-B^3,\\
&P_4:=\psi_{4,E_T}(0,0)=CB^5, \quad \quad P_5:=\psi_{5,E_T}(0,0)=-(C-B)B^8, \\
&P_6:=\psi_{6,E_T}(0,0)=-B^{12}(C^2-B+C), \quad \quad P_7:=\psi_{7,E_T}(0,0)=B^{16}(C^3-B^2+BC) .
\end{align*}
A computation shows $P_k = (q_3 / q_2^3)^{k^2 - 1} Q_k$ for $k < 5$. This must then be true for all $k$ since both sequences $P_k$ and $Q_k$ satisfy the
recurrence relations~(\ref{RedOdd}),(\ref{RecEven}), are preserved under scaling (defined in Section~\ref{SectionModUnits}).
From Equation~(\ref{tildeq}),
\begin{equation}
	P_k	= \left(\frac{q_3}{q_2^3}\right)^{k^2 - 1} Q_k
		=  \left(\frac{q_3}{q_2^{8/3}}\right)^{k^2 - 1} \tilde{Q}_k
		= \left(\tilde{q}_3\right)^{k^2 - 1} \prod_{d | k} \tilde{q}_d
		= F_3^{\lfloor k^2/3 \rfloor} \prod_{3 < d | k} F_d.
		\label{relation}
\end{equation}
In particular, the multiplicative group $\left\langle D,-B,P_4,\ldots,P_k \right\rangle$ equals $\left\langle F_2,\ldots,F_k\right\rangle$.
Streng defined $F_k$ for $k>3$ to be $P_k$ with all factors of $P_j$ with $j<k$
removed, Equation~(\ref{relation}) makes this precise.

Since $\psi_{k,E_T}(P)=0$ if and only if $P$ has order dividing $k$, we see $F_k(P)=0$ if and only if $P$ has exact order $k$. 
As mentioned in Section \ref{SectionModUnits}, the polynomial $F_N$ is a model for the modular curve $X_1(N)$ for $N>3$. The Tate normal form~(\ref{Tatenormal}) is only defined for $N>3$, so \cite{DerickxvanHoeij} used $x,j$ coordinates
to construct $F_2$ and $F_3$. Rewritten in terms of $B,C$ they are $F_2=B^4D^{-1}$ and $F_3=-B$. We can now state the main result of \cite{Streng}
\begin{theorem}\label{Strengmain}
\cite[Theorem 1.1]{Streng}, \cite[Conjecture~1]{DerickxvanHoeij}. The modular units of $X^1(N)$ are given by $\QQ^*$ times the free abelian group on  $B,D,F_4,F_5,\dots, F_{\lfloor N/2\rfloor+1}$,
or equivalently, $F_2,\ldots,F_{\lfloor N/2\rfloor+1}$.
\end{theorem}
 
Streng gives $P_k$ explicitly in terms of Siegel functions. 
\begin{lemma}\label{Strengformulas}
\cite[Lemma 3.3]{Streng} For all $k\in \ZZ- N\ZZ$
\[P_k=\left(\frac{H_1^2H_3}{H_2^3}\right)^{k^2-1}\frac{H_k}{H_1} \ \ {\rm and} \ \ D=\left(\frac{H_1^2H_3}{H_2^3}\right)^{12}  H_1^{12}.\]
\end{lemma}
Defining $\tilde{H}_k := H_k / H_1^{k^2}$, we get
\begin{equation} \label{PkF3F2}
	P_k = \left(\frac{\tilde{H}_3}{\tilde{H}_2^3}\right)^{k^2-1} \tilde{H}_k, \ \
	F_3 = P_2 = \frac{\tilde{H}_3^3}{\tilde{H}_2^8}, \ \ \text{and} \ \ 
	F_2 = \frac{P_2^4}{D} = \tilde{H}_2^4 .
\end{equation}
%
%
%
Setting $t = c/N$,
equation~(\ref{Ord}) gives
\begin{equation} \label{tH1}
	{\rm uord}_c\left(\tilde{H}_k\right) = {\rm uord}_c\left( H_k \right) - k^2 {\rm uord}_c\left(H_1\right) = \frac12 \left(\BB_2\hspace{-.07 cm}\left(\hspace{-.07 cm}\left\{ kt  \right\}\hspace{-.07 cm}\right)
	- k^2 \BB_2\hspace{-.07 cm}\left(\hspace{-.07 cm}\left\{ t \right\}\hspace{-.07 cm}\right)\right).
\end{equation}
We say that a function $f: [0,1/2] \rightarrow \RR$ is $k$-{\em piecewise linear} if it is continuous and $f''(t) = 0$ for all $t \not\in \frac1k \ZZ$.
Two $k$-piecewise linear functions coincide if and only if they have:  the same initial value $f(0)$, the same initial slope $f'(0^+)$, and the same change in slope
at each $t \in \frac1k \ZZ$. These three conditions hold
for the right-hand sides of~(\ref{tH1}) and~(\ref{tH}) and thus:
\begin{equation} \label{tH}
 {\rm uord}_c( \tilde{H}_k ) = \frac12 \left( (k^2-k)t - \frac16(k^2-1) \right) +  k \sum_{0 < i < k/2} \left( \min\left(t, \frac{i}k\right) - t\right).
\end{equation}
Applying~(\ref{tH}) to $F_2$ and $F_3$ in~(\ref{PkF3F2}) produces the unweighted order functions $v_2(t)$ and $v_3(t)$ in Theorem~\ref{Main}.

To verify $v_k(t)$ for the remaining $k > 3$, let $\tilde{v}_k(t)$ be the unweighted order function of $\tilde{q}_k$, i.e. $\tilde{v}_k(t)$  
is the right hand side of equation~(\ref{valqt})
without the factor $s_k$.  So $\tilde{v}_2(t) = 0$, $\tilde{v}_3(t) = \frac13 v_3(t)$ and $\tilde{v}_k(t) = v_k(t)$ for $k>3$.
The unweighted order function for $\tilde{Q}_k = \prod_{d | k} \tilde{q}_d$ according
to Theorem~\ref{Main} and Remark~\ref{rqt} is
\begin{equation} \label{lasteq}
	\sum_{d | k} \tilde{v}_d(t) = \sum_{d|k} \sum_{0<c'<d/2} \bfw_{c'}(d) m_{{c'}/d}(t) = k \sum_{0<i<k/2} m_{{i}/k}(t),
\end{equation}
where $m_a(t) = \min(t,a) - 4 a(1-a)t$. We also used that $k$ is the sum of $\bfw_{c'}(d)$, taken over all $0<c'<d/2$ with $d|k$ and $c'/d = i/k$.

Applying~(\ref{tH}) to $\tilde{Q}_k = \tilde{H}_k / \tilde{H}_2^{(k^2-1)/3}$ gives the same result. To see this, note that $m_{i/k}(t)$, which contains $\min(t, i/k)$,
appears in \eqref{lasteq} with the same coefficient $k$
as the coefficient of $\min(t, i/k)$ in \eqref{tH}.
The terms $\frac16(k^2-1)$ from \eqref{tH} cancel out for $\tilde{H}_k / \tilde{H}_2^{(k^2-1)/3}$ but then the remaining terms $(\cdots) t$ in  \eqref{tH},\eqref{lasteq} must also match by the $\int_0^{1/2} = 0$ argument from Remark~\ref{rqt}.
This confirms $\tilde{v}_k(t)$ and thus $v_k(t)$ for the remaining $k$. This gives a second proof for
most (except the case $N | k$, see lemma~\ref{Strengformulas}) of the reformulation of Theorem~\ref{Main} given in Remark~\ref{rqt}.


\appendix 

\section{Proof of Primitivity}

\begin{proposition}\label{primitive}
The polynomial $P_k$
is primitive in $\ZZ[B,C]$.
\end{proposition}


\begin{proof}
Order the monomials  lexicographically with the following rule
\[B^{n_1}C^{n_2}<B^{m_1}C^{m_2} \ \   {\rm when} \ \ n_1<m_1 \ {\rm or} \ (n_1=m_1 \ {\rm and} \ n_2<m_2). \]
If $R\in \QQ[B,C]$, let $M(R)$ denote the smallest monomial of $R$. For example, if $R=3B^2C^5+B^3C$, then $M(R)=3B^2C^5$. A key property is $M(R_1R_2)=M(R_1)M(R_2)$.

Let $c_k$ denote $\lceil k/3\rceil$. It is enough to prove that
\begin{equation}\label{MLemma}
M(P_k)=(-1)^{c_k}(-B)^{\lfloor k^2/3\rfloor} C^{c_k(c_k-1)/2}
\end{equation}
since it shows that $P_k$ has at least one coefficient equal to $\pm 1$.

We will prove \eqref{MLemma} by induction. First, a direct verification shows that~\eqref{MLemma} holds for $k=1, 2, 3,4$. Suppose now that $k$ is even, and write $l=\frac{k}{2}$. Recall the recursion relation~\eqref{RecEven}
\[P_k=\frac{P_l}{P_2}\left(P_{l+2}P_{l-1}^2-P_{l-2}P_{l+1}^2\right).\]


The smallest monomial of the first summand $P_{l+2}P_{l-1}^2$ is
\[(-1)^{c_{l+2}+2c_{l-1}}(-B)^{\lfloor (l+2)^2/3\rfloor+2\lfloor (l-1)^2/3\rfloor} C^{c_{l+2}(c_{l+2}-1)/2+c_{l-1}(c_{l-1}-1)}.\]
For the second summand $-P_{l-2}P_{l+1}^2$ it is
\[(-1)^{c_{l-2}+2c_{l+1}}(-B)^{\lfloor (l-2)^2/3\rfloor+2\lfloor (l+1)^2/3\rfloor} C^{c_{l-2}(c_{l-2}-1)/2+c_{l+1}(c_{l+1}-1)}.\]
When $l\equiv 1 \bmod 3$, the second summand has the smallest monomial, and when $l\equiv 2 \bmod 3$, the first summand has the smallest monomial.
When $3\mid l$, we have to consider the exponent of $C$. In this case, the first summand is the smallest.

In each case, verifying Equation~\eqref{MLemma} is straightforward. 
For example, when $l\equiv 0 \bmod 3$ we have
\[\lfloor l^2/3\rfloor+\lfloor (l+2)^2/3\rfloor+2\lfloor (l-1)^2/3\rfloor=\frac{4l^2+3}{3}=\lfloor k^2/3\rfloor+1,\]
and
\[c_l(c_l-1)/2+c_{l+2}(c_{l+2}-1)/2+c_{l-1}(c_{l-1}-1)=\frac{l}{3}\left(\frac{2l-3}{3}\right)=\frac{c_k}{2}(c_k-1).\]

Now suppose $k$ is odd and write $k=2l+1$.  Recall the recursion relation~\eqref{RedOdd} 
\[P_k=P_{l+2}P_{l}^3-P_{l-1}P_{l+1}^3.\]
For the first summand, the smallest monomial is
\[(-1)^{c_{l+2}+3c_{l}}(-B)^{\lfloor (l+2)^2/3\rfloor+3\lfloor l^2/3\rfloor} C^{c_{l+2}(c_{l+2}-1)/2+3c_{l}(c_{l}-1)/2},\]
and for the second summand it is
\[(-1)^{c_{l-1}+3c_{l+1}}(-B)^{\lfloor (l-1)^2/3\rfloor+3\lfloor (l+1)^2/3\rfloor} C^{c_{l-1}(c_{l-1}-1)/2+3c_{l+1}(c_{l+1}-1)/2}.\]
When $l\equiv 0\bmod 3$, the first summand has the smaller monomial; when $l\equiv 1 \bmod 3$, considering the exponent of $C$ shows the
second summand has the smaller monomial; and when $l\equiv 2 \bmod 3$, the first summand has the smaller monomial.

Verifying Equation~\eqref{MLemma} is again straightforward for each case.
For example, when $l\equiv 1\bmod 3$
\[\lfloor (l-1)^2/3\rfloor+3\lfloor (l+1)^2/3\rfloor=\frac{4l^2+4l+1}{3}=\lfloor k^2/3\rfloor,\]
and
\[c_{l-1}(c_{l-1}-1)/2+3c_{l+1}(c_{l+1}-1)/2=\frac{4l^2-2l-2}{18}=\frac{c_k}{2}(c_k-1).\]
Repeating these computations for the remaining cases proves the proposition.
\end{proof}


Recall that $-B = F_3 = \tilde{q}_3^3$ and $-C = F_4$. From~\eqref{relation} we find that $\tilde{Q}_{k \setminus 3}$ from Section~\ref{SectionModUnits} is $P_k/(-B)^{\lfloor k^2/3 \rfloor}$ which is primitive in $\ZZ[B,C]=\ZZ[F_3,F_4]$
by equation~\eqref{MLemma}.


\bibliographystyle{abbrvnat}

\newpage
\section{Notation}\label{NotationSection}

\begin{table}[h!]
\centering
\renewcommand{\arraystretch}{1.5} 
\begin{tabular}{|c|c|c|}
\hline
\textbf{Notation} & \textbf{Brief Definition} & \textbf{References} \\ \hline
$\bfp$ & a Puiseux expansion above $s=0$ & page \pageref{Preliminaries} \\ \hline
$v_{\bfp}$ & discrete valuation associated to $\bfp$ & page \pageref{Preliminaries} \\ \hline
$\bfe_{\bfp}$ & smallest $e$ with $\bfp\in \CC \LPar s^{1/e}\RPar$ \ \ ($=v_{\bfp}(s)$) & page \pageref{Preliminaries} \\ \hline
$k_{\bfp}$ & residue field associated to $\bfp$ & page \pageref{Preliminaries} \\ \hline
$\bff_{\bfp}$ & $\left[k_{\bfp}:\QQ\right]$ & page \pageref{Preliminaries} \\ \hline
$\bfn_{\bfp}$ & $\bfe_{\bfp}\cdot \bff_{\bfp} = \left[ \QQ((s))[\bfp] : \QQ((s)) \right]$ & page \pageref{Preliminaries} \\ \hline
$l_s({\bfp})$ & dominant term of $\bfp$ & page \pageref{Preliminaries} \\ \hline
$E_{\epsilon}$ & $y^2=x(x-\epsilon)(x-1)$, with $0 < \e \ll 1$ & page \pageref{A} \\ \hline
$\omega_1,\omega_2$ & periods of $E_{\epsilon}$ & page \pageref{A} \\ \hline
$W$ & $(\CC/\Lambda)/\pm$ where $\Lambda = \ZZ \omega_1 + \ZZ \omega_2 $ & page \pageref{A} \\ \hline
$W(N)$ & elements of order $N$ in $W$ & page \pageref{A} \\ \hline
$C(N)$ & the Cartan, $\left\{0,1,\dots, \lfloor N/2 \rfloor \right\}$ & page \pageref{A} \\ \hline
$W_i(N)$ & $i^{\text{th}}$ Galois orbit $\subset W(N)$ \ \  ($i \in C(N)$)  & page \pageref{A} \\ \hline
$C_i(N)$ & $i^{\text{th}}$ Galois orbit $\subset \{x(P) \, | \, P {\rm \ order \ } N\}$ & page \pageref{A} \\ \hline
$\bfn_i$ & $| W_i(N) | =  | C_i(N) | = \bfe_i \cdot \bff_i$ & page \pageref{A} \\ \hline
$E$ & $y^2 = x^3  + a x + b$, $a= - 3 j_0$, $b = - 2 j_0$ &  pages \pageref{SecDiv}, \pageref{SectionModUnits}\\ \hline
$j_0$, $j$ & $j_0=j/(j-1728)$, \ $j=j$-invariant of $E$& page \pageref{SectionModUnits}  \\  \hline
$s$ & $s = 1/j$, \ roots($s$) = $\{$cusps of $X_1(N)\}$ & page \pageref{F2} \\  \hline
$E[=\hspace{-2pt}k]$   & $\{$points on $E$ of exact order $k\}$ & page \pageref{SecDiv}  \\ \hline
$m_k$       & $m_k=\#$ points of exact order $k$ & page \pageref{SectionModUnits}   \\ \hline
$Q_k$ & division polynomial of $E$ 
& pages \pageref{SecDiv}, \pageref{prodab} \\ \hline
$q_k$ &  roots$(q_k) =  \{ x(P)\ |\ P \text{ has order } k\}$,
\ \ $Q_k=\prod_{d\mid k}q_k$ & pages \pageref{SecDiv}, \pageref{prodab}  \\  \hline
$\tilde{Q}_k$ & rescaling of $Q_k$ to make it unique, \ $\tilde{Q}_k=Q_k/Q_2^{(k^2-1)/3}$ & page \pageref{SectionModUnits} \\ \hline
$\tilde{q}_k$ & rescaling of $q_k$, \ $\tilde{Q}_k=\prod_{d\mid k}\tilde{q}_d$  & page \pageref{SectionModUnits}  \\ \hline
$\{F_k\}$ & basis of modular units,  $F_k  \in \QQ(x, j_0) = \QQ(x, j) = \QQ(x, s)$ & \cite{DerickxvanHoeij}, \cite{Streng} \\ \hline
$F_2, F_3$ & $F_2 = q_2^4/ \left(1728 j_0^2 (j_0-1)\right)$, \ $F_3 = \tilde{q}_3^3=q_3^3/q_2^8$ & pages \pageref{4}, \pageref{SectionModUnits}  \\ \hline
$F_k$, $k>3$ & $F_k = \tilde{q}_k = q_k/q_2^{m_k/3}$ & page \pageref{SectionModUnits}  \\ \hline
$v_k(t)$ & piecewise linear function, gives $\divisor\left(F_k\right)$ on any $X_1(N)$ & page \pageref{Main} \\ \hline
\end{tabular}
\vspace{.05 in}
\caption{Summary of notation for Sections \ref{Preliminaries}-\ref{4}} 
\label{Notation24}
\end{table}


\newpage

\begin{table}[h!] 
\centering
\renewcommand{\arraystretch}{1.6} 
\begin{tabular}{|c|c|c|}
\hline
\textbf{Notation} & \textbf{Brief Definition} & \textbf{References} \\ \hline
$v(t)$ & 
$v(t)=v_7(t)-v_8(t)$, gives  $\divisor\left(F_7/F_8\right)$ & page \pageref{notationref} \\ \hline
$m(t)$ & $\max(v(t),0)$ & page \pageref{notationref} \\ \hline
$B_0(N)$ & degree of $F_7/F_8$ & page \pageref{notationref} \\ \hline
$B_1(N)$ & upper bound for $B_0(N)$ & page \pageref{notationref} \\ \hline
$I_i$ & intervals where $m(t)$ is linear & page \pageref{Defmi} \\ \hline
$m_i(t)$ & linear function equal to $m(t)$ restricted to $I_i$ & page \pageref{Defmi} \\ \hline
$\bfe_{\scriptscriptstyle{\left[ {a \atop c} \right]}}(N)$ & width of cusp on $X_1(N)$ with representative $\left[\begin{smallmatrix} a\\ c \end{smallmatrix}\right]$  & page \pageref{WidthNotRef} \\ \hline
$g_{(a_1,a_2)}$ & Siegel function associated to $(a_1,a_2)\in\QQ^2-\ZZ^2$ & page \pageref{defSiegel} \\ \hline
$\BB_2$ & second Bernoulli polynomial, $x^2-x+\frac{1}{6}$  & page \pageref{SecondBernoulli} \\ \hline
$\{\bullet\}$ & fractional part, $\bullet-\lfloor \bullet\rfloor$ & page \pageref{ordg} \\ \hline
$[\bullet]_1$ & first entry in a vector & page \pageref{ordg} \\ \hline
$H_k$ & Siegel function on $X_1(N)$, $g_{(0,k/N)}$, $k\in \ZZ-N\ZZ$ & page \pageref{ordg} \\ \hline
$\uord_c\left(H_k\right)$ & ``unweighted order'' of $H_k$ at $c\in C(N)$ 
& page \pageref{Ord} \\ \hline
$E_T$ & Tate normal form:  $Y^2+(1-C)XY-BY=X^3-BX^2$ & page \pageref{Tatenormalform} \\ \hline
$B, C$ & $E \leadsto E_T$ gives $B=-F_3$, $C=-F_4$, $\QQ(B,C)=\QQ(x,j)$ & pages \pageref{Tatenormalform}, \pageref{Strengformulas} \\ \hline
$\Psi_{k,E_T}$ & $k^{\text{th}}$ division polynomial of $E_T$ & page \pageref{Strengformulas} \\ \hline
$P_k$ &\hspace{-.1 cm} $P_k=\Psi_{k,E_T}(0,0)\hspace{-.1 cm} \in \ZZ[B,C]$, factors: $F_3$\hspace{-.05 cm} and \hspace{-.05 cm}$\{F_d : 2\hspace{-.05 cm}\neq \hspace{-.05 cm}d\hspace{-.05 cm}\mid \hspace{-.05 cm}k\}$ & page \pageref{Strengformulas} \\ \hline
$\tilde{H}_k$ & scaled Siegel function, $H_k/H_1^{k^2}$ & page \pageref{PkF3F2} \\ \hline
\end{tabular}
\vspace{.05 in}
\caption{Summary of notation for Sections \ref{5}-\ref{6}} 
\label{Notation56}
\end{table}

\end{document}